\documentclass{article}
\usepackage[utf8]{inputenc}
\usepackage{amssymb}
\usepackage{amsthm}
\usepackage{amsfonts,amsmath,amsbsy,latexsym,mathrsfs
}
\usepackage{authblk}
\newtheorem{thm}{Theorem}
\newtheorem{cor}[thm]{Corollary}
\newtheorem{lem}[thm]{Lemma}

\newtheorem{defi}[thm]{Definition}

\title{Energy of a graph and Randic index}

\author{Gerardo Arizmendi and Octavio Arizmendi\thanks{G.A. received support from CONACyT grant 256126.  O.A. received support from Conacyt Grant CB-2017-2018-A1-S-9764 an from the European Union's Horizon 2020 research and innovation programme under the Marie Sk\l{}odowska-Curie grant agreement No 734922.\\~}}

\date{ \today}

\begin{document}

\maketitle

\begin{abstract}
We prove that, for any graph $G$, its graph energy is at least twice the Randic index. We show that equality holds if and only if  $G$ is the  union of complete bipartite graphs.
\end{abstract}

\section{Introduction}

Among graph descriptors used in mathematical chemistry, two of them play a rather important role and have attracted a lot of interest: The graph energy  \cite{Gut,GR,LSG} and the Randic Index \cite{Ran}. The first one comes from spectral graph theory, while the second one is of vertex-degree-based  nature.

There are plenty of inequalities for each of these descriptors. However, so far, no direct relation between them has been established.\footnote{ To the best of our knowledge the only result in this direction is \cite{Jah}, but the result there involves different quantities.}

In this paper, we prove a very precise relation between this two quantities, the graph energy is at least twice the Randic index. Our proof is based on \emph{vertex energy}, a refinement of the graph energy which was introduced in \cite{AJ}, in order to understand locally the properties of the latter.  We also show that equality holds if and only if  the graph consists of a union of complete bipartite graphs.

Apart from this introduction the paper contains two short sections. In Section 2, we introduce the notions of energy and vertex energy and show that the product of energies of two adjacent vertices is greater or equal to one. In Section 3 we prove the main result of the paper.

\section{An inequality for vertex energy}

\subsection{Graph energy and vertex energy} 

Let us consider a graph $G=(V,E)$ with vertex set $V=\{v_1,...,v_n\}$ and adjacency matrix $A\in M_n(\mathbb{R})$.  The \emph{energy of the graph} \cite{Gut,GR,LSG} is given by 
\begin{equation*}
\mathcal{E}(G)=\sum^n_{i=0} |\lambda_i|. 
\end{equation*}
A graph of size $n$ is called hypoenergetic if $\mathcal E(G)<n$, \cite{Gut2} and orderenergetic if $E(G)=n$ \cite{AGK}.

If, for a matrix $M$, we denote its trace by $Tr(M)$, and its absolute value $(MM^*)^{1/2}$, by $|M|$, then the energy of $G$ is given by
\begin{equation*}
\mathcal{E}(G)= Tr(|A(G)|)=\displaystyle\sum_{i=1}^{n}|A(G)|_{ii}.
\end{equation*}
In \cite{AJ}, the energy of a vertex is introduced.
\begin{defi}
The energy of the vertex $v_i$ with respect to $G$, which is denoted by $\mathcal{E}_G(v_i)$, is given by
\begin{equation}
  \mathcal{E}_G(v_i)=|A(G)|_{ii}, \quad\quad~~~\text{for } i=1,\dots,n,
\end{equation}
where $|A|=(AA^*)^{1/2}$ and $A$ is the adjacency matrix of $G$.
\end{defi}

In this way the energy of a graph is given by the sum of the individual energies of the vertices of $G$,
\begin{equation*}
  \mathcal{E}(G)=\mathcal{E}_G(v_1)+\cdots+\mathcal{E}_G(v_n),
\end{equation*}
and thus the energy of a vertex is a refinement of the energy of a graph.

The following lemma  tells us how to calculate the energy of a vertex in terms of the eigenvalues and eigenvectors of $A$.
\begin{lem}{\rm \cite{AFJ}} \label{L1} Let $G=(V,E)$ be a graph with vertices $v_1,...,v_n$. Then
\begin{equation}\label{Equa2}
  \mathcal{E}_{G}(v_i)=\displaystyle\sum_{j=1}^{n}p_{ij}|\lambda_j|,\quad i=1,\ldots,n
\end{equation}
where $\lambda_j$ denotes the $j$-eigenvalue of the adjacency matrix of $A$ and the weights $p_{ij}$ are given by $p_{ij}=u_{ij}^{2}$ where $U = (u_{ij})$ is  the orthogonal matrix whose columns are given by the eigenvectors of $A$.
\end{lem}

\subsection{Energy of adjacent vertices} 

The next theorem is fundamental for the proof of our main result and is of independent interest.

\begin{thm}\label{T1}
Let $v_i$ and $v_j$ be connected vertices of a simple (undirected) graph G. Then ${\cal E}(v_i){\cal E}(v_j)\geq1$. 
\end{thm}
\begin{proof}

 Let $A(G)$ be the adjacency matrix of $G$. Then we can write $A(G)=UDU^t$, where $U=(u_{kl})$ is orthogonal and $D=(d_{kl})$ is a diagonal matrix such that $d_{kk}=\lambda_k$, with  $\lambda_k$'s the eigenvalues of $A(G)$. Hence, by Lemma \ref{L1},  ${\cal E}(v_i)=\sum_ku_{ik}^2|\lambda_k|$ and ${\cal E}(v_j)=\sum_ku_{jk}^2|\lambda_k|$. Moreover $A(G)_{ij}=\sum_ku_{ik}u_{jk}\lambda_k$. Since $v_i$ and $v_j$ are connected then $A(G)_{ij}=1$.

Now consider $$v=(u_{i1}\sqrt{|\lambda_1|},\dots,u_{in}\sqrt{|\lambda_n|})$$ 
and
$$w=(u_{j1}sign(\lambda_1)\sqrt{|\lambda_1|},\dots,u_{jn}sign(\lambda_n)\sqrt{|\lambda_n|})$$ 
then 
$$\left<v,w\right>^2=(\sum_ku_{ik}u_{jk}\lambda_k)^2=1$$
$$||v||^2=\sum_ku_{ik}^2|\lambda_k|={\cal E}(v_i)$$
$$||w||^2=\sum_ku_{jk}^2|\lambda_k|={\cal E}(v_j)$$
which proves the assertion by the Cauchy-Schwarz inequality.
\end{proof}


By the use of AM-GM inequality we observe that:
\begin{cor}
Let $v_i$ and $v_j$ be connected vertices of a simple (undirected) graph G. Then ${\mathcal E}(v_i)+{\mathcal E}(v_j)\geq2$.
\end{cor}

Recall that a matching in a graph is a set of edges that are pair-wise non-adjacent.
\begin{cor}
If $G$ has a matching of size $k$, then $\mathcal E (G)\geq2k$. In particular, if $G$ has a complete matching then it is not hypoenergetic.
\end{cor}

\section{Randic index and energy of graphs}

We now prove the main theorem of the paper. For a vertex $v$, we will denote by $deg(v)$, the degree of the vertex $v$, i.e. the number of neighbours of $v$. Let us recall that for a graph $G=(V,E)$ the Randic index is given by
 $$R(G)=\sum_{(v,w)\in E} \frac{1}{\sqrt{{deg(v)deg(w)}}}.$$

\begin{thm}\label{mainT}
Let $G$ be a graph with graph energy $\mathcal{E}(G)$ and Randic index $R(G)$ then $\mathcal{E}(G)\geq 2R(G)$.  
\end{thm}

\begin{proof}
Let $G=(V,E)$. For an edge $e=(v,w)$ define $\mathcal{E}(e)=\mathcal{E}(v)/deg(v)+\mathcal{E}(w)/deg(w)$.
Then, on one hand, \begin{eqnarray*}
\sum_{e\in E}\mathcal{E}(e)&=&\sum_{e\in E}\left( \frac{\mathcal{E}(v)}{deg(v)}+\frac{\mathcal{E}(w)}{deg(w)}\right)\\&=&\sum_{(v,w)\in E} \frac{\mathcal{E}(v)}{deg(v)}+\sum_{(v,w)\in E}\frac{\mathcal{E}(w)}{deg(w)}
\\&=&\frac{1}{2}\sum_{v\in V}\sum_{w\sim v} \frac{\mathcal{E}(v)}{deg(v)}+\frac{1}{2} \sum_{w\in V(G)}\sum_{v\sim w} \frac{\mathcal{E}(w)}{deg(v)}\\
&=&\frac{1}{2}\sum_{v\in V} \mathcal{E}(v)+\frac{1}{2} \sum_{w\in V} \mathcal{E}(w)
=\sum_{v\in V} \mathcal{E}(v)=\mathcal{E}(G).
\end{eqnarray*}

On the other hand, by the classical AM-GM inequality, if $e=(v,w)$,
$$\mathcal{E}(e)=
\frac{\mathcal{E}(v)}{deg(v)}+\frac{\mathcal{E}(w)}{deg(w)}\geq 2\sqrt{\frac{\mathcal{E}(v)\mathcal{E}(w)}{deg(v)deg(w)}}\geq2 \frac{1}{\sqrt{{deg(v)deg(w)}}},$$
where we used Theorem \ref{T1} in the last inequality.

Finally, summing over $e\in E(G)$ we obtain the desired inequality

$$\mathcal{E}(G)=\sum_{e\in E(G)} \mathcal{E}(e)\geq 2 \sum_{(v,w)\in E(G)} \frac{1}{\sqrt{{deg(v)deg(w)}}}=2R(G)$$

\end{proof}

Recall that a graph is called regular if all its vertices have the same degree. Since for any regular graph $R(G)=n/2$, we recover the following results from \cite{GFPR}.
\begin{cor}
Regular graphs are non-hypoenergetic.
\end{cor}

A similar result may be derived for bipartite semi-regular graphs. These are graphs whose vertex set can be divided in two subsets $V_1$ and $V_2$, such that all the vertices in $V_1$ have degree $d_1$, all the vertices in $V_2$ have degree $d_2$ and any edge connects a vertex in $V_1$ with a vertex in $V_2$.
\begin{cor}
 Let $G$ be a bipartite semi-regular graph, with sizes $n_1$ and $n_2$ and degrees $d_1$ and $d_2$ then
$$\mathcal{E}(G)\geq n_1\frac{\sqrt{d_1}}{\sqrt{d_2}}+n_2\frac{\sqrt{d_1}}{\sqrt{d_2}}. $$
\end{cor}

In order to characterize the graphs where equality $\mathcal E(G)=2R(G)$ is attained, we need to use the following theorems from \cite{GFPR}, which also state when we have equality in the previous corollaries.

\begin{thm} \cite[Theorem 3]{GFPR}\label{T9}
 Let $G$ be a regular graph of size $n$, of degree $r$ , $r > 0$ . Then
\begin{equation}\label{Equ reg}
\mathcal E(G)\geq n.
\end{equation}
Equality in \eqref{Equ reg} is attained if and only if every component of $G$ is isomorphic to the
complete bipartite graph $K_{r,r}$.
\end{thm}

\begin{thm}
\cite[Theorem 5]{GFPR} \label{T10} Let $G$ be a bipartite semi-regular with sizes $n_1$ and $n_2$ and degrees $d_1$ and $d_2$ then
\begin{equation}\label{Equ semireg} 
\mathcal{E}(G)\geq n_1\frac{\sqrt{d_1}}{\sqrt{d_2}}+n_2\frac{\sqrt{d_1}}{\sqrt{d_2}}. \end{equation}
Moreover, equality in \eqref{Equ semireg} is attained if and only if every component of $G$ is isomorphic to the complete bipartite graph $K_{d_1,d_2}$.
\end{thm}

\begin{thm}
 $\mathcal{E}(G)= 2R(G)$ if and only if $G$ is a disjoint union of complete bipartite graphs.
\end{thm}

\begin{proof}
Since the Randic index and the graph energy are additive with respect to connected components, we may assume without loss of generality that $G$ is connected.

If $\mathcal{E}(G)= 2R(G)$ then all the inequalities used in the proof of theorem \ref{mainT} must become equalities. In particular, for vertices $v$ and $w$, we must satisfy 

$$\frac{\mathcal{E}(v)}{deg(v)}+\frac{\mathcal{E}(w)}{deg(w)}= 2\sqrt{\frac{\mathcal{E}(v)\mathcal{E}(w)}{deg(v)deg(w)}}$$
which by AM-GM, is only possible if $$\frac{\mathcal{E}(v)}{deg(v)}=\frac{\mathcal{E}(w)}{deg(w)}.$$
This, together with the condition $\mathcal{E}(v)\mathcal{E}(w)=1,$ implies that
$$\mathcal{E}(v)=\frac{\sqrt{deg(v)}}{\sqrt{deg(w)}}.$$
Since the left hand-side does not depend of $w$, then all the neighbours of $v$ have the same degree as $w$. Similarly, all the neighbours of $w$ have the same degree as $v$. Continuing in this way, one sees that we have two possibilities.

If $deg(v)=deg(w)$, then all the vertices have the same degree, i.e. $G$ is a $d$-regular graph.  In this case $R(G)=n/2$ and in order that $\mathcal{E}(G)=2R(G)=n$, by Theorem \ref{T9}, $G$ must be $K_{r,r}$ for some $r$.

Otherwise, if $deg(v)\neq deg(w)$,  we
can divide the set of vertices in two sets $A$ and $B$, such that all the vertices in $A$ have degree $d_A$ and all the vertices in $B$ have degree $d_B$, and any edge connects a vertex in $A$ with a vertex in $B$. This means that $G$ is a bipartite semi-regular graph.

For this bipartite semi-regular graph, if we denote by $n_A$ and $n_B$ the sizes of $A$ and $B$, respectively, the Randic index may be calculated by
$$R(G)=n_A\frac{\sqrt{d_A}}{\sqrt{d_B}}=n_B\frac{\sqrt{d_B}}{\sqrt{d_A}}.$$
Now, in order for $\mathcal E(G)=2R(G)$ we need 
$$\mathcal{E}(G)=n_A\frac{\sqrt{d_A}}{\sqrt{d_B}}+n_B\frac{\sqrt{d_B}}{\sqrt{d_A}},$$ 
which is only possible if $G$ is a complete bipartite graph $K_{d_A,d_B}$, by  Theorem \ref{T10}.

\end{proof}

\noindent Department of Actuarial Sciences, Physics and Mathematics. Universidad de las Am\'{e}ricas Puebla. San Andr\'{e}s Cholula, Puebla. M\'{e}xico.\\\noindent Email: \emph{gerardo.arizmendi@udlap.mx}
	\\~
	
\noindent Department of Probability and Statistics. Centro de Investigaci\'{o}n en Matem\'aticas, Guanajuato, M\'{e}xico. \\\noindent Email: \emph{octavius@cimat.mx}
\end{document}